\documentclass[12pt,a4paper]{article}
\usepackage{latexsym}
\usepackage{graphicx}
\DeclareSymbolFont{AMSb}{U}{msb}{m}{n}
\DeclareSymbolFontAlphabet{\mathbb}{AMSb}

\newcommand{\Nset}{\mathbb{N}}
\newcommand{\Qset}{\mathbb{Q}}
\newcommand{\Rset}{\mathbb{R}}
\newcommand{\Tset}{\mathbb{T}}
\newcommand{\Zset}{\mathbb{Z}}
\newcommand{\Cyl}{\mathcal{C}}

\newtheorem{thm}{Theorem}
\newtheorem{lem}[thm]{Lemma}
\newtheorem{remark}[thm]{Remark}
\newenvironment{rem}{\begin{remark} \rm}{\end{remark}}

\newcommand{\proof}{\noindent{\bf Proof. }\ignorespaces}
\newcommand{\proofof}[1]{\noindent{\bf Proof of #1. }\ignorespaces}
\newcommand{\qed}{\hfill$\Box$}
\newcommand{\mqed}{\quad\quad\Box}
\newcommand{\eqref}[1]{\textup{(\ref{#1})}}

\newcommand{\abs}[1]{\left\vert#1\right\vert}
\newcommand{\Or}[1]{\mathop\mathrm{O}\nolimits\left(#1\right)}
\newcommand{\Lnorm}[1]{\left\Vert#1\right\Vert_{L^2}}
\newcommand{\D }[1]{D#1}
\newcommand{\DD }[1]{D^2 #1}
\newcommand{\dc }[1]{\partial_c #1}

\newcommand{\Dc }[1]{\frac{\partial#1}{\partial c}}
\newcommand{\diff}{\,\mathrm{d}}
\newcommand{\LE}{\chi^+}
\newcommand{\e}{{\rm e}}
\newcommand{\ic}{\mskip2mu{\rm i}\mskip1mu}
\newcommand{\To}{\longrightarrow}

\title{A method of reduction for invariant curves of
       quasiperiodically forced maps}

\author{Amadeu Delshams\\
        Departament de Matem\`{a}tiques\\
        Universitat Polit\`{e}cnica de Catalunya\\
        Diagonal 647, 08028 Barcelona\\
        \texttt{Amadeu.Delshams@upc.edu}
   \and Rafael Ortega\\
        Departamento de Matem\'{a}tica Aplicada\\
        Facultad de Ciencias\\
        Universidad de Granada, 18071 Granada\\
        \texttt{rortega@ugr.es}}

\begin{document}

\maketitle

\begin{abstract}
The existence of translated curves for quasiperiodically forced maps is established, under very mild regularity hypotheses, for rotation numbers of constant type. Among the translated curves, the invariant curves are characterized as the solutions of a scalar bifurcation equation, from which their existence, stability as well as bifurcation can be easily described.
\end{abstract}

\section{Introduction}

This paper deals with the existence of invariant curves of skew-product
dynamical maps $(r_1,\theta_1)=f_\varepsilon(r,\theta)$ of the form
\begin{equation}
 \label{Eq:FE}
 \begin{array}{rcl@{\qquad}l}
 r_1&=&r+\varepsilon F(r,\theta; \varepsilon),& r\in\Rset\\
 \theta_1&=&\theta+2\pi\alpha,& \theta\in\Tset=\Rset /2\pi\Zset ,
 \end{array}
\end{equation}
defined on the cylinder $\Cyl:=\Rset \times\Tset$, where $\varepsilon$ is
a small parameter.

For an irrational frequency $\alpha$, such kind of maps are called
\emph{quasiperiodically forced maps} and arise in a number of well-known
examples.

As a motivation, let us mention the so-called 
\emph{quasiperiodically forced Arnold circle map}
 \begin{equation}
 \label{Eq:qpfAmap}
 \begin{array}{rcl@{\qquad}l}
 x_1&=&x+\omega+k\sin x+b\sin\theta,& x\in\Rset\, ,\\
 \theta_1&=&\theta+2\pi\alpha,& \theta\in\Tset=\Rset /2\pi\Zset\, ,
 \end{array}
 \end{equation}
 which has been studied by many authors as a model of quasiperiodically forced oscillators and as a paradigm for the investigation of strange non-chaotic attractors (SNA, see, for instance, 
 \cite{Jager09,FiguerasH12,Simo18} and references therein).

 This paper focuses on finding smooth and invariant invariant curves for map~(\ref{Eq:FE}) in a simple way, which will also allow us to address the problem of their existence and stability. In this direction, let us begin by noting that to find these invariant curves in the system above~(\ref{Eq:qpfAmap}),
 in \cite{GlendinningW99,GlendinningFPS00} the authors dealt with small $\omega$ and $k$, and thus introduced expansions in a small parameter $\varepsilon$
 \[
 k=\varepsilon,\quad
 \omega=\omega_1\varepsilon+\omega_2\varepsilon^2+\cdots,\quad
 b=b_0+b_1\varepsilon +b_2\varepsilon^2+\cdots,
 \]
and searched for invariant curves of the map above~(\ref{Eq:qpfAmap}) of the type
 \begin{equation}
 \label{Eq:expansion}
 x=G(\theta )=G_0
 (\theta )+G_1 (\theta )\varepsilon +G_2 (\theta )\varepsilon^2
 +\cdots,
 \end{equation}
where all functions $G_i$ are $2\pi$-periodic. After substituting
this expansion into the mapping~\eqref{Eq:qpfAmap}, the function $G_0$ can
be explicitly computed for $\alpha \not\in \Zset$. Namely,
 \begin{equation}\label{expl}
 G_0(\theta)=g_0
 -\frac{b_0}{2\sin(\pi\alpha)}\cos(\theta-\pi\alpha),
 \end{equation}
where $g_0$ is a constant to be adjusted later (in the computation of
$G_1$). The functions $G_1$, $G_2,\dots,$ can be determined in a recursive
way. They are solutions of linear functional equations of the type
 \begin{equation}
 \label{Eq:LinSmallDiv}
 G_n (\theta +2\pi \alpha )=G_n (\theta )+p_n (\theta),
 \end{equation}
where $p_n$ depends upon the parameters $\omega_1,\dots,\omega_n$,
$b_1,\dots,b_n$, and the previous functions $G_0 ,G_1,\dots,G_{n-1}$
(see~\cite{GlendinningW99} for more details of these formal computations).

The solutions to equation~(\ref{Eq:LinSmallDiv}) can be obtained using
Fourier analysis and, for $\alpha$ irrational, the condition
 \[
 \int_0^{2\pi} p_n (\theta)\diff\theta =0
 \]
becomes necessary and sufficient for their solvability. Sometimes, these
solutions are only formal because small divisors appear. To overcome this
objection one has to assume that $\alpha$ has good arithmetic properties.
However, even when $\alpha$ is the golden mean, it is not obvious how to
prove the convergence of expansion~\eqref{Eq:expansion}.

Before coming back to the general case~\eqref{Eq:FE}, we notice that
map~\eqref{Eq:qpfAmap} can also be written in this form, simply by performing
the change of variables $x=r+G_0 (\theta)$ where $G_0$ is given
by~\eqref{expl}. Indeed, under this change, map~\eqref{Eq:qpfAmap} becomes
\begin{equation}
 \label{Eq:FEOmegaIntro}
 \begin{array}{rcl}
 r_1&=&r+\varepsilon\left(\omega_1+
      \widetilde F (r,\theta;\varepsilon)\right),\\
 \theta_1&=&\theta+2\pi\alpha,
 \end{array}
\end{equation}
where $\widetilde F (r,\theta;\varepsilon)=\sin(r+G_0 (\theta))
+b_1\sin\theta+\Or{\varepsilon}$.

In this paper we propose a different approach for studying the existence
of invariant curves. First, we shall show that the cylinder $\Cyl$ can be
foliated by \emph{translated curves} of the map~\eqref{Eq:FE}.

By a translated curve, we understand a curve of the type $r=G(\theta )$
which is mapped under map~\eqref{Eq:FE} onto $r_1 =G(\theta_1 )-\lambda$,
for some $\lambda \in \Rset$, that is, $G(\theta)$ and $\lambda$ are
solutions of the equation
\begin{equation}\label{Eq:FuncEqIntro}
 G(\theta+2\pi\alpha)=G(\theta)
 +\varepsilon F\left(G(\theta),\theta;\varepsilon\right)+\lambda\, .
\end{equation}

Once we know that such a foliation exists, it is enough to discuss the
equation $\lambda =0$ among all translated curves to get invariant curves.

Let us be more concrete. We will assume that $F\in
C^{m}(\Cyl\times[0,1])$ in equation~\eqref{Eq:FE}, for $m\ge 3$. More
precisely, we assume that
\begin{equation}
 \label{Eq:boundF}
 \abs{\partial^{\ell_1}_\theta\partial^{\ell_2}_r
 \partial^{\ell_3}_\varepsilon
 F(r,\theta;\varepsilon)} \leq k_F^{(m)}
\end{equation}
for all $(r,\theta;\varepsilon)\in \Cyl\times [0,1]$,
$\ell=(\ell_1,\ell_2,\ell_3)\in\Nset^3$,
$\abs{\ell}:=\ell_1+\ell_2+\ell_3 \leq m$. We also assume that the frequency $\alpha$
has good arithmetic properties, to be precise, $\alpha$ is assumed to be of
\emph{constant type}; see~\eqref{Eq:ConsType}. Under these two hypotheses, our
main result (Theorem~\ref{Thm:Main}) is that for $\varepsilon$ small
enough ($\varepsilon\in[0,\varepsilon^*)$, with $\varepsilon^*$ depending
only on $k_F^{(3)}$ and $\alpha$)), there exists a foliation of $\Cyl$ by
translated curves of the map~\eqref{Eq:FE}. More precisely, there exists a
$C^{m-2}$ map $\psi: \Cyl\times [0,\varepsilon^*) \to \Rset$ of the form
$\psi(c,\varepsilon,\theta)=c+\Or{\varepsilon}$, such that, for every
$(c,\varepsilon)\in\Rset\times[0,\varepsilon^*)$,
$r=\psi(c,\theta,\varepsilon)$ is a translated curve of~\eqref{Eq:FE},
satisfying
\[
 c=\frac 1{2\pi}\int^{2\pi}_0\psi(c,\theta,\varepsilon)
 \diff \theta\, .
 \]
Moreover, for every fixed $\varepsilon\in[0,\varepsilon^*)$,
$(c,\theta)\mapsto (\psi(c,\theta), \theta)$ is a diffeomorfism in the
cylinder $\Cyl$ and therefore the family of translated curves
$r=\psi(c,\theta,\varepsilon)$ defines a foliation of the cylinder.

The most delicate aspect of the method presented here is the construction
of translated curves. To this end, instead of dealing with standard KAM
rapidly convergent iterations or invoking ``hard'' implicit function theorems in scaled spaces, we propose
a simple iteration scheme where a linear functional-differential equation is solved
at each step. This procedure is certainly more artificial than the method
proposed in \cite{GlendinningW99}, but the iteration scheme is very easy
to state and we can justify its convergence in a rigorous way.

Our approach is based on Herman's ideas in \cite{Herman82,Herman86}
and is also related to the papers of de la Llave~\cite{Llave83,Llave86}.

Once a lemma for the solution of a linear difference equation in a suitable Sobolev space and a suitable frequency $\alpha$ (Lemma~\ref{Lem:Herman}, due to Herman) is taken for granted, the proof
of Theorem~\ref{Thm:Main} is completely self-contained.
We note again that one could also
employ an approach based on Newton's Method as in KAM theory, and even
consider a multi-dimensional $\theta \in \Tset ^{d}$. This would relax
the conditions on $\alpha$ but it probably would require more regularity
on $F$ and, of course, this would complicate the proof a lot.

Once the existence of a foliation of translated curves
$r=\psi(c,\theta,\varepsilon)$ is proved, the existence of invariant
curves among them is reduced to searching zeros of the
``translation'' function $c\mapsto\lambda=\lambda(c,\varepsilon)$,
which has the form $\lambda(c,\varepsilon)= -\varepsilon
\Phi(c,\varepsilon)$, with
 \begin{equation}
 \label{Eq:BifFuncIntro}
 \Phi(c,\varepsilon)=\frac1{2\pi}\int^{2\pi}_0
 F(\psi(c,\theta,\varepsilon),\theta;\varepsilon) \diff \theta\, .
 \end{equation}
Equivalently, for $\varepsilon\in[0,\varepsilon^*)$, the invariant curves
of map~\eqref{Eq:FE} are \emph{precisely} given by
$r=\psi(c^*,\theta,\varepsilon)$ for every zero $c^*$ of the real
function $c\mapsto\Phi(c,\varepsilon)$.

Therefore, the search for invariant curves is reduced to finding the zeros
of $\Phi$. The reader who is familiar with the methods of bifurcation
theory will notice that this strategy is reminiscent of Lyapunov-Schmidt
and alternative methods (see~\cite{Hale80}). Actually,
equation~\eqref{Eq:FuncEqIntro} for the translated curves can be seen as
the \emph{auxiliary equation}, while $\Phi(c,\varepsilon)=0$ becomes the
\emph{bifurcation equation}.

From the study of the scalar equation $\Phi(c,\varepsilon)=0$, global and
local results about the existence of invariant curves can be easily
deduced. For example, each simple zero $c^*$ of $\Phi$ will give rise to a
hyperbolic invariant curve, attractor for $\dc\Phi(c^*,\varepsilon)<0$, or
repeller otherwise. More interestingly, multiple zeros of $\Phi$ give
rise to bifurcating invariant curves.

On the other hand, for a \emph{circle} map~\eqref{Eq:FE}, that is, for a
function $F$ periodic in the variable $r$, it turns out that $\Phi$ is
also a periodic function of $c$, and the equation $\Phi(c,\varepsilon)=0$
will typically have a finite number of zeros. For instance, for the
map~\eqref{Eq:FEOmegaIntro}, the bifurcation equation takes the form
$\omega_1+\widetilde \Phi(c,\varepsilon)=0$, with a $2\pi$-periodic in
$c$ function $\widetilde \Phi(c,\varepsilon)$ given by
\[
 \widetilde \Phi(c,\varepsilon)=\frac1{2\pi}\int^{2\pi}_0
 \widetilde F(\psi(c,\theta,\varepsilon),\theta;\varepsilon)
 \diff \theta\, .
\]
In particular, this implies that the set of $\omega_1$'s for which there
is at least one invariant curve, that is, $\Phi^{-1} (-\omega_1)$ is
non-empty, is a compact interval. This is a fact that seems to be given
for granted in~\cite{GlendinningW99}. Moreover, if
$[\omega_\ast,\omega^\ast]$ is such an interval, then there exist at
least two invariant curves if $\omega_1 \in (\omega_\ast,\omega^\ast)$.
In principle, the degenerate situation $\omega_\ast=\omega^\ast$ cannot
be excluded. In such a case, all translated curves would become invariant
simultaneously and the map~\eqref{Eq:FEOmegaIntro} should be conjugate to
$r_1=r$, $\theta_1=\theta+2\pi\alpha$.

Finally, one can also obtain computable conditions for the existence of
invariant curves of map~\eqref{Eq:FE}, by expanding
$\Phi(c,\varepsilon)=\Phi_0(c)+\varepsilon\Phi_1(c)+\cdots$, with
 \[
 \Phi_0(c):=\frac1{2\pi}\int^{2\pi}_0 F(c,\theta;0) \diff \theta\, ,
 \]
and looking for finite-order zeros of $\Phi_0$.

\section{Translated curves}
\label{Sec:translated}
The methodology based on proving the existence of translated curves to subsequently search for invariant curves was initiated by R\"ussmann~\cite{Russmann70}, later Hermann~\cite{Herman82,Herman86} implemented it for diffeomorphisms with frequency of constant type, and it has recently been used in other contexts by Gonz\'alez \emph{et al.}~\cite{GonzalezHL14},
Massetti~\cite{Massetti18} and Pello~\cite{Pello22}.

In our context, given a function $\psi\in C (\Tset )$, its graph will be
denoted by
\[
\Gamma=\Gamma_\psi=\{(\psi(\theta),\theta): \,\, \theta\in\Tset \}\, .
\]
Then a curve $r=\psi(\theta)$ is said to be a \emph{translated curve}
if there exists a number $\lambda$ (called the \emph{translation number}) such that
\[
f_\varepsilon(\Gamma)=\Gamma-(\lambda,0)\, .
\]
 \begin{figure}
 \centering
  \includegraphics[width=0.85\textwidth]{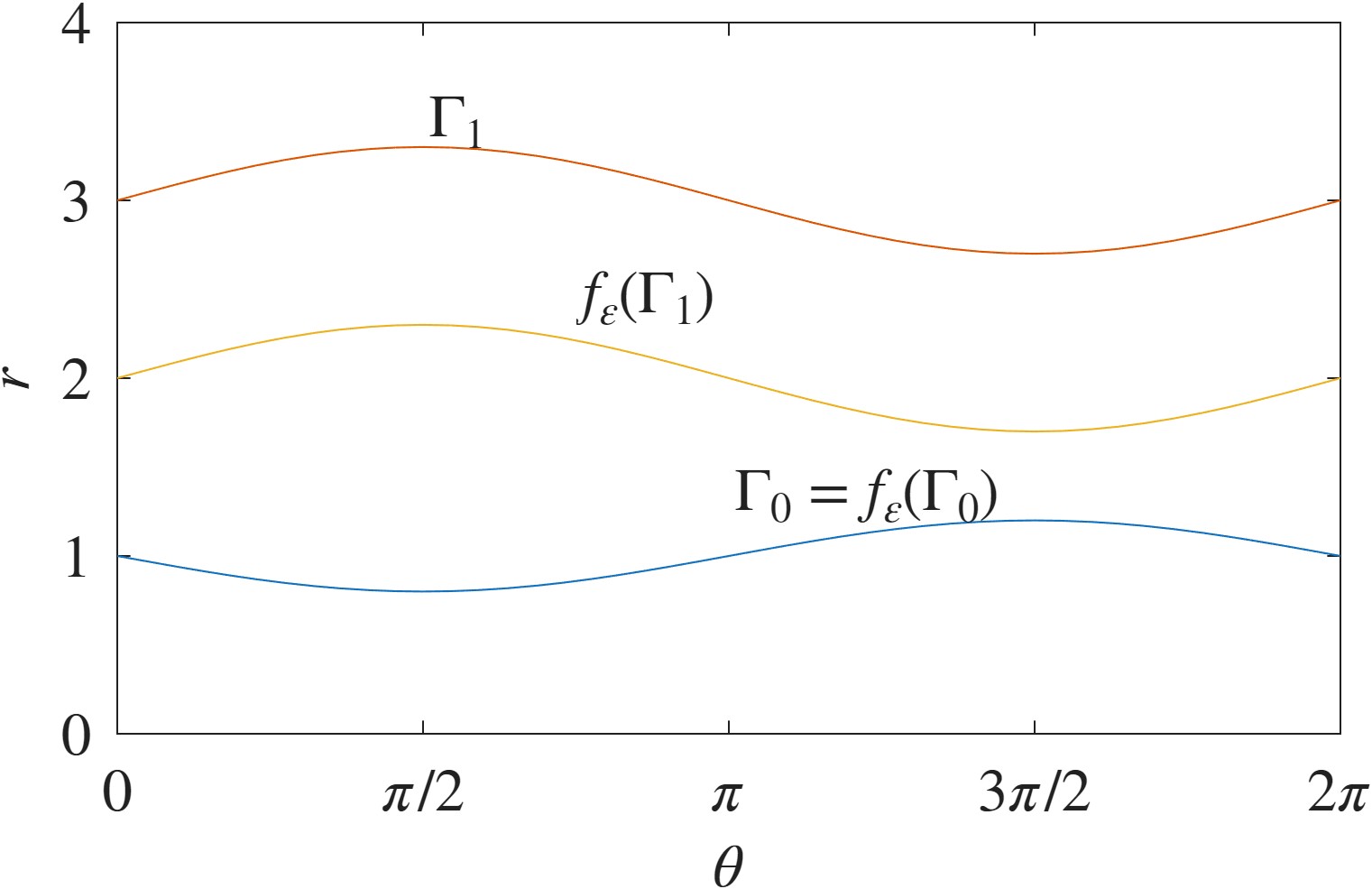}
 \caption{Representation of translated and invariant curves: $\Gamma_0$ is an invariant curve and $\Gamma_1$ a translated one.}
 \label{Fig:TransCurve}
 \end{figure}
 
For $\lambda=0$, the translated curve will be called \emph{invariant
curve} (see Figure~\ref{Fig:TransCurve}). For small $\varepsilon$,
$f_\varepsilon$ is a diffeomorphism and there is a one-to-one
correspondence between translated curves and continuous solutions of
the functional equation
 \begin{equation}\label{Eq:FuncEq}
 \psi(\theta+2\pi\alpha)=\psi(\theta)
 +\varepsilon F\left(\psi(\theta),\theta;\varepsilon\right)+\lambda\, .
 \end{equation}
Notice that the unknown is the couple
$(\psi,\lambda)\in C(\Tset )\times \Rset $.

When the function $F$ does not depend on $r$ and $\varepsilon$,
$F(r,\theta;\varepsilon )=p(\theta )$, equation~(\ref{Eq:FuncEq})
takes the form of one of the simplest and best known small
divisors problems: the linear equation
 \[
 \psi(\theta+2\pi\alpha)
 =\psi(\theta)+\varepsilon p(\theta)+\lambda \ ,
 \]
where $p$ is a given function. In this case the translation  number
must be $\lambda=-\varepsilon\bar p$, where
 $\displaystyle  \bar p:=1/(2\pi)\int^{2\pi}_0 p$, and it is well known (see, for
instance, \cite{KatokH95}) that one cannot guarantee the existence
of continuous solutions unless some conditions on the arithmetic of the frequency
$\alpha$ are imposed.

We shall assume that $\alpha$ is of \emph{constant type}; that is,
 \begin{equation}
 \label{Eq:ConsType}
 \exists \delta >0 :\abs{\alpha-p/q}\geq\delta/q^2,
 \qquad\forall\, p/q\in\Qset .
 \end{equation}
More information on these numbers can be seen in \cite{Herman82}.
They are characterized by having a bounded continued fraction. An
example of number of constant type is the golden mean
$(1+\sqrt{5})/2$, whose continued fraction contains only the
number one.

\begin{thm}
\label{Thm:Main} Assume $\alpha$ is of constant
type~\textup{(\ref{Eq:ConsType})}.  Then there exists
$\varepsilon^\ast=\varepsilon^\ast(\alpha,k_F^{(3)})>0$ such that for
$0\le\varepsilon<\varepsilon^\ast$ and $c\in\Rset $, then the map
$f_\varepsilon$ given in~\textup{(\ref{Eq:FE})} has a unique
translated curve $r=\psi_c(\theta)$, with $\psi_c\in H^2(\Tset )$
and $\Lnorm{ \DD \psi_c}\leq 1$, satisfying
 \[
 c=\frac 1{2\pi}\int^{2\pi}_0\psi_c(\theta) \diff \theta\, .
 \]

Moreover, the map $(c,\theta)\in\Rset\times\Tset\To
\psi_c(\theta)\in\Rset$ is $C^1$.
\end{thm}

\begin{rem}\label{sobolev}
$H^2(\Tset )$ denotes the Sobolev space of functions $u$ on $\Tset $
having a weak second derivative $\DD u$ belonging to $L^2(\Tset )$.
If $u$ has Fourier expansion $\sum \widehat u_n\e ^{in\theta}$, this is
equivalent to saying that
 \[
 \Lnorm{\DD u}:=2\pi \sum \abs{n}^4\abs{\widehat u_n}^2<\infty\, .
 \]
In an analogous way we can define $H^1 (\Tset )$ as the space of
functions satisfying
 \[
 \Lnorm{\D u}:=2\pi \sum \abs{n}^2\abs{\widehat u_n}^2<\infty\, .
 \]
Next we list some well-known properties of Sobolev spaces
(see, for instance, \cite{Brezis83,Herman86}) which will be
employed throughout the paper.
\begin{enumerate}
\item[(i)] $H^2 (\Tset )$ is
contained in $C^1 (\Tset )$ and, for each $u\in H^2 (\Tset )$ with
zero mean value,
\[
\abs{u(\theta )}+\abs{\D u(\theta )}\leq C\Lnorm{\DD u}
\mbox{ for all }\theta \in \Rset .
\]
\item[(ii)]
Assume that $\{ u_n \}$ is a sequence in $H^1 (\Tset )$ such that
$\Lnorm{\D u_n }$ is bounded and $\Lnorm{u_n } \to 0$. Then $u_n
(\theta )\to 0$ uniformly in $\theta$.
\item[(iii)]
Let $\{ u_n \}$ be a sequence in $H^2 (\Tset )$ and
$u\in H^1 (\Tset )$. Assume that $u_n$ and $u$ have constant mean
value, say $\displaystyle{1\over 2\pi}\int_0^{2\pi} u_n ={1\over 2\pi}
\int_0^{2\pi} u=c$, and
 \[
 \Lnorm{\DD u_n } \leq \Delta, \quad \Lnorm{\D u_n -\D u} \to 0.
 \]
Then
$u\in H^2 (\Tset )$, $\Lnorm{\DD u } \leq \Delta$, and
 \[
 u_n (\theta )\to u(\theta ),\; \; \D u_n (\theta )\to \D u(\theta ),
 \mbox{ uniformly  in } \theta.
 \]
\end{enumerate}
\end{rem}

\begin{rem}
By property~(i), the function $\psi_c$ provided by
Theorem~\ref{Thm:Main} has a classical first order derivative. The
proof of Theorem~\ref{Thm:Main} will lead to the estimate
 \[ \label{Eq:Estimate} \Lnorm{\DD \psi_c}
 =\Or{\varepsilon} \quad\mbox{as}\quad \varepsilon\to 0\,.
 \]
Again, in view of property~(i) of the previous remark,
$\psi_c(\theta)=c+\Or{\varepsilon}$, $\D \psi
_c(\theta)=\Or{\varepsilon}$, uniformly in $\theta$.
\end{rem}
\begin{rem}
Although it seems possible to get theoretical estimates of the value of $\varepsilon^*(\alpha,k_F^{(3)})$ of Theorem~\ref{Thm:Main} from those made by Herman in \cite{Herman86}, such estimates are typically very unrealistic. Nevertheless, since the (constructive) Theorem~\ref{Thm:Main} relies on Sobolev spaces, the \emph{Sobolev criterion} (see \cite{CallejaC10,CallejaCL13, CallejaCGL24} and references therein) appears to be fully adequate to provide real estimates on the breakdown of translated curves.
\end{rem}
To prove Theorem~\ref{Thm:Main} we shall construct an iteration scheme which
will converge to $\psi_c$. To motivate this scheme we perform a
formal computation. Taking a derivative in (\ref{Eq:FuncEq}) we
are led to
 \begin{equation}\label{der}
 \D \psi(\theta+2\pi\alpha)
 =\Bigl(1+\varepsilon F_r(\psi(\theta),\theta;\varepsilon)\Bigr)
 \D \psi(\theta)+ \varepsilon F_\theta(\psi(\theta),\theta;\varepsilon)\, ,
 \end{equation}
where $F_r$, $F_\theta$ denote, respectively, the derivative of $F$
with respect to $r$, $\theta$.

We are looking for $2\pi$-periodic solutions of this
functional-differential equation and the advantage with respect to
the previous equation is that $\lambda$ has disappeared. To exploit
linearity as much as possible we will begin with
 \[
 \psi_0=c
 \]
and for $n\geq 0$, we shall use the iteration
 \begin{eqnarray}
 \D \psi_{n+1}(\theta+2\pi\alpha)
 &=&\Bigl(1+\varepsilon F_r(\psi_n(\theta),\theta;\varepsilon)\Bigr)
 \D \psi_{n+1}(\theta)\nonumber\\
 &&\quad + \varepsilon
 F_\theta(\psi_n(\theta),\theta ;\varepsilon)+\nu_n\, ,
 \label{Eq:DefPsin}\\
 \frac 1{2\pi}\int^{2\pi}_0\psi_{n+1}(\theta) \diff \theta&=&c\, ,\nonumber
 \end{eqnarray}
where $\nu_n\in\Rset $ must be determined step by step.

\begin{rem}
The introduction of the number $\nu_n$ seems rather artificial but
it is required, because if we let $\nu_n=0$ it might happen that the
scheme is not well defined for some $n$ (more on this point in
section~\ref{SubsecNun}). Anyway, the numbers $\nu_n$ will be
unimportant as $n\to\infty$. Actually, if $\psi_n$ converges in
$C^1(\Tset )$ to a function $\psi$ then, passing to the limit in the
scheme one would arrive at
 \[
 \D \Bigl(\psi(\theta+2\pi\alpha)-\psi(\theta)
 -\varepsilon F(\psi(\theta),\theta;\varepsilon)\Bigr)=\nu\, ,
 \]
where $\nu$ is any accumulation point of the bounded sequence
$\{\nu_n\}$. Since the function under derivation is periodic we
conclude that $\nu=0$ and so $\psi$ solves
equation~(\ref{Eq:FuncEq}) for some $\lambda \in \Rset$.
\end{rem}
To prove the convergence of $\{\psi_n\}$ we shall employ some
results by Herman on the linear difference equation
 \begin{equation}\label{Eq:LinDE}
 \varphi(\theta+2\pi\alpha)-a(\theta)\varphi(\theta)=p(\theta)+\nu\, ,
 \quad \int^{2\pi}_0\varphi (\theta)\diff \theta=0\, ,
 \end{equation}
where the unknowns are $\varphi$ and $\nu$.

We now state a lemma that is a reformulation of the results of
\cite[chapter VII, section 9]{Herman86}. There, the complete proof of this lemma is based on Herman's strategy, which consists of taking derivatives and logarithms in equation~(\ref{Eq:LinDE}), so that the compositions are simplified, in order to reduce it to a standard small divisor equation like equation~(\ref{Eq:LinSmallDiv}). Furthermore, all constants are computed explicitly.

\begin{lem}
\label{Lem:Herman}
Assume that $\alpha$ is of constant type~\textup{(\ref{Eq:ConsType})}
and let $a\in H^2(\Tset )$ with
 \[
 0<a_-\leq a(\theta)\leq a_+<\infty\, ,\quad\forall\, \theta\in\Tset \, .
 \]
Then, for each $p\in H^1(\Tset )$, there exists a unique solution
$(\varphi,\nu)\in L^2(\Tset )\times\Rset $ of~\textup{(\ref{Eq:LinDE})}
satisfying
 \[
 \Lnorm{\varphi}\leq C\Lnorm{\D p}\, .
 \]
Here $C$ only depends on $a_{\pm}$, $\alpha$ and $\Lnorm{\DD a}$.

Moreover, there exists a constant $\Delta=\Delta(a_{\pm},\alpha)>0$,
such that if
 \[
\Lnorm{\DD a}\leq \Delta\, ,
 \]
and $p\in H^2 (\mathbb{T})$ then $\varphi\in H^1(\Tset )$ and
 \[
 \Lnorm{\D \varphi}\leq C^\ast \Lnorm{\DD p}\, ,
 \]
with $C^\ast=C^\ast(a_{\pm},\alpha)$.
\end{lem}

\proofof{the existence of $\psi_c$ in Theorem~\ref{Thm:Main}}
As a first step we prove that, if $\varepsilon$ is small,
then $\psi_n$ as defined inductively
in~(\ref{Eq:DefPsin}) satisfy
$\psi_n\in H^2(\Tset )$ and
 \[
 \Lnorm{\DD \psi _n}\leq 1\, .
 \]

This is clear for $\psi_0$. By induction, let us assume it for
$\psi_n$ and let us prove it for $\psi_{n+1}$. We apply
lemma~\ref{Lem:Herman} with $a(\theta)=1+\varepsilon F_r(\theta,
\psi_n(\theta);\varepsilon)$ and $p(\theta) =\varepsilon
F_\theta(\theta,\psi_n(\theta);\varepsilon)$. Let us fix $a_-=1/2$,
$a_+=3/2$, and we assume $a_-\leq a\leq a_+$ and $\parallel
D^2a\parallel_{L^2}\leq \Delta$. This is achieved by restricting the
size of $\varepsilon$. Then equation~(\ref{Eq:LinDE}) has a unique
solution in $H^1(\Tset ) \times\Rset $ and we can construct
$\psi_{n+1}$ as the primitive of this solution. Thus $\psi_{n+1} \in
H^2(\Tset )$ and
 \[
 \Lnorm{\DD \psi _{n+1}}\leq C^\ast\Lnorm{\D^2 p} \leq\varepsilon
 K\, .
 \]
Since the constant $K$ only depends on $k_F^{(3)}$, we restrict again the
size of $\varepsilon$ so that $\Lnorm{\DD \psi _{n+1}}\leq 1$.

Next we are going to prove that
 \[
 \Lnorm{\D \psi _{n+1}-\D \psi _n} \leq K\varepsilon
 \Lnorm{\D \psi _n-\D \psi _{n-1}}\, , \quad n\geq 1\, ,
 \]
where $K$ is a fixed number. To do this we notice that
$\varphi:=D(\psi_{n+1}-\psi_n)$ and $\nu=\nu_n-\nu_{n-1}$ satisfy
equation~(\ref{Eq:LinDE}) with
 \begin{eqnarray*}
 a(\theta)&=&
 1+\varepsilon F_r(\theta,\psi_n(\theta);\varepsilon),\\
 p(\theta)&=&
 \varepsilon\Bigl(F_r(\theta,\psi_{n}(\theta);\varepsilon)
 - F_r(\theta,\psi_{n-1}(\theta);\varepsilon)\Bigr)\D \psi _n(\theta)\\
 &&\mbox{} +
 \varepsilon\bigl(F_\theta(\theta,\psi_{n}(\theta);\varepsilon)
 - F_\theta(\theta,\psi_{n-1}(\theta);\varepsilon)\Bigr)\, .
 \end{eqnarray*}
 From here we deduce that
 \[
 \Lnorm{\D \psi _{n+1}-\D \psi _n}\leq C\Lnorm{\D p}\leq
 K\varepsilon\Lnorm{\D \psi _{n}-\D \psi _{n-1}}\, .
 \]
Here we have used that $\D \psi _n(\theta)$ is uniformly bounded
since it can be controlled in terms of $\Lnorm{\DD \psi _n}\leq 1$
(this is a consequence of property~(i) of Remark~\ref{sobolev}), and
that
 \[
 \psi_{n}(\theta)-\psi_{n-1}(\theta)=
 \int^\theta_{\theta^\ast}D(\psi_n-\psi_{n-1})(\Theta) \, \diff \Theta,
 \]
where $\theta^\ast$ is such that
$\psi_{n}(\theta^\ast)-\psi_{n-1}(\theta^\ast)=0$
($\psi_{n}-\psi_{n-1}$ has zero mean).

The previous estimate implies that the series
 $\sum \Lnorm{\D \psi_{n+1} -\D \psi_n }$ is majorized by
 $\sum (K\varepsilon )^n <\infty$. This implies that there exists
$\psi \in H^1 (\Tset )$ with $\displaystyle{1\over 2\pi} \int_0^{2\pi} \psi =c$
and such that $\Lnorm{\D \psi_n -\D \psi }  \to 0$. We can now apply
property~(iii) of Remark~\ref{sobolev} to conclude that $\psi$ is in
$H^2 (\Tset )$ and $\Lnorm{\D^2 \psi } \leq \varepsilon K\leq 1$. We
notice that $\psi_n (\theta )$ and $\D \psi_n (\theta )$ converge to
$\psi (\theta )$ and $\D \psi (\theta )$ in a uniform sense. This
gives us the existence of a translated curve.

The uniqueness is obtained by a repetition of the previous arguments.
Actually, if $\psi$ and $\psi^\ast$ would produce translated
curves with the same average then $\D {(\psi - \psi^\ast)}$ would
be a solution again of a linear difference
equation~(\ref{Eq:LinDE}) and we could obtain an estimate of the
type
 \[
 \Lnorm{\D \psi-\D {\psi^\ast}}\leq K\varepsilon
 \Lnorm{\D \psi-\D {\psi^\ast}}\, ,
 \]
leading to $\psi=\psi^\ast$ for small $\varepsilon$.
\qed
\section{Dependence with respect to parameters}

First of all we notice that $\psi=\psi_c$ depends continuously with
respect to $c$. Indeed, we can use the strategy of the previous
proof to obtain an estimate of the type
 \begin{equation}\label{err}
 \Lnorm{\D \psi _{c_1}-\D \psi _{c_2}} \leq K\varepsilon
 \abs{c_1-c_2} ,
 \end{equation}
where $\psi_c$ is the translated curve with
$\displaystyle c=1/(2\pi)\int^{2\pi}_0\psi_c$. From here one can deduce that the
function $\psi=\psi(\theta,c)$ is Lipschitz-continuous and strictly
increasing in the variable $c$. This fact follows from the previous
estimate and the identity
 \[
 \psi_{c_1}(\theta)-\psi_{c_2}(\theta)=c_1-c_2
 +\int^\theta_{\theta^\ast}
 D(\psi_{c_1}-\psi_{c_2})(\Theta) \, \diff \Theta\, ,
 \]
where $\theta^\ast$ is a point for which
$\psi_{c_1}(\theta^\ast)-\psi_{c_2}(\theta^\ast)=c_1-c_2$. As a
consequence of the previous remarks we can say that the mapping
$(c,\theta)\longmapsto (\psi_c(\theta), \theta)$ is a homeomorphism
of the cylinder $\Cyl$ and that the family of translated curves
defines a foliation. In particular, notice that they do not
intersect.

The function $\psi$ depends on several parameters,
$\psi=\psi(\theta,c,\varepsilon)$, or even
$\psi=\psi(\theta,c,\varepsilon,\omega,b,\dots)$, if we go back to
the original mapping~(\ref{Eq:qpfAmap}). Next we show how to compute first order
derivatives. We shall analyze $\dc \psi$.

If we proceed in a purely formal way and take derivatives with
respect to $c$ in equation~(\ref{der}), then $\delta=\dc \psi$
should verify
 \begin{eqnarray}
 \D\delta (\theta+2\pi\alpha) &=&
 \Bigl(1+\varepsilon F_r(\psi_c(\theta),\theta;\varepsilon)\Bigr)
 \D \delta (\theta )
 + \varepsilon\Bigl(F_{rr}(\psi_c(\theta),\theta;\varepsilon)
 \D \psi _c(\theta)\nonumber\\
 &&\mbox{}+F_{\theta r}(\psi_c(\theta),\theta; \varepsilon)\Bigr)\delta(\theta),
 \label{Eq:Var}\\ \frac{1}{2\pi}\int^{2\pi}_0\delta(\theta) \diff
 \theta&=&1\, .\nonumber
 \end{eqnarray}
To be rigorous, we first show that (\ref{Eq:Var}) has a unique
solution in $H^1 (\Tset )$. Actually (\ref{Eq:Var}) can be
rewritten as
 \[
 \D \Bigl( \delta (\theta +2\pi \alpha )
 -(1+\varepsilon F_r (\psi_c(\theta),\theta; \varepsilon))
 \delta (\theta )\Bigr)=0.
 \]
This is equivalent to
 \[
 \delta (\theta +2\pi \alpha )
 -(1+\varepsilon F_r (\psi_c(\theta),\theta; \varepsilon))
 \delta (\theta )=\nu
 \]
for some $\nu \in \Rset$. Now the function $\varphi (\theta )=\delta
(\theta )-1$ is a solution of~(\ref{Eq:LinDE}) with $a=1+\varepsilon
F_r(\psi_c,\cdot ; \varepsilon )$ and $p=\varepsilon
F_{r}(\psi_c,\cdot ;\varepsilon ) $.
 From Lemma \ref{Lem:Herman} we deduce that $\delta \in H^1
(\Tset )$ is unique and
 \begin{equation}
 \label{otra} \Lnorm{\D \delta } =\Or{\varepsilon }.
 \end{equation}
Once we know that the candidate for derivative is well determined we
define, for a given $h\neq 0$,
 \[
 \Delta_h(\theta)=\frac{1}{h}\bigl(\psi_{c+h}(\theta)-\psi_c(\theta)\bigr)\,
 \]
and we must prove that $\Delta_h$ converges to $\delta$ as $h\to
0$. To do this we notice that the function $\rho_h =\Delta_h
-\delta$ has zero mean value and satisfies
 \begin{equation}\label{dde}
 \D \Bigl( \rho_h (\theta +2\pi \alpha )-(1+\varepsilon F_r
 (\psi_c(\theta),\theta; \varepsilon))\rho_h (\theta )\Bigr) =q_h
 (\theta )
 \end{equation}
where
\begin{eqnarray*}
 q_h &=&\varepsilon\left[\left(\!\frac{F_r (\psi_{c+h},\cdot;\varepsilon )
 -F_r(\psi_c,\cdot;\varepsilon)}{h}\!\right)\D \psi _{c+h}
 -F_{rr}(\psi_c,\cdot ;\varepsilon )\Delta_h \D \psi_c \right]\\
 &&\mbox{}+\varepsilon \left[
 \!\frac{F_\theta(\psi_{c+h},\cdot;\varepsilon)
 -F_\theta(\psi_c,\cdot ;\varepsilon)}{h}
 \!-F_{\theta r} (\psi_c,\cdot;\varepsilon )\Delta_h \right].
 \end{eqnarray*}
 
 From (\ref{err}) we deduce that
 \[
 \Lnorm{\D \psi_{c+h} -\D \psi_c } \leq K\varepsilon |h|.
 \]
In particular $\Lnorm{\D \Delta_h }$ is bounded independently of $h$.
This implies that $q_h$ satisfies $\Lnorm{q_h} \to 0$ as $h\to 0$.
By integrating (\ref{dde}) we get
 \[
 \rho_h (\theta +2\pi \alpha )-a(\theta )\rho_h (\theta )
 =p_h (\theta )+\nu_h,
 \]
with
$a(\theta)=1+\varepsilon F_r(\theta,\psi_{c}(\theta);\varepsilon)$,
where $p_h$ is the primitive of $q_h$ with zero mean value. Notice
that~(\ref{dde}) guarantees that $q_h$ has zero mean value.

Again lemma~\ref{Lem:Herman} implies that
 \[
 \Lnorm{\rho_h } \leq C\Lnorm{q_h } .
 \]
Thus, for small $\varepsilon$, $\Lnorm{\rho_h } \to 0$ and, since
 $\Lnorm{\D \rho_h }=\Lnorm{\D \Delta_h - \D \delta}$ is bounded,
independently of $h$, we can apply property~(ii) of
remark~\ref{sobolev} to conclude that $\rho_h (\theta )\to 0$
uniformly in $\theta$. This proves that $\partial_c \psi$ exists
everywhere and satisfies (\ref{Eq:Var}), and finishes the proof of
Theorem~\ref{Thm:Main}.

\section{Invariant curves: existence and stability}
\label{Sec:ExistStab}

Let $\lambda=\lambda_c$ be the translation number of the translated
curve $r=\psi_c(\theta)$. Invariant curves will appear for those
values of $c$ for which $\lambda=0$. In more analytic terms we
consider the function
 \begin{equation}
 \label{Eq:BifFunc}
 \Phi(c):=\frac1{2\pi}\int^{2\pi}_0
 F(\psi_c(\theta),\theta;\varepsilon) \diff \theta\, .
 \end{equation}
We know from Theorem~\ref{Thm:Main} that $\Phi$ is $C^1$. Since
$\lambda_c$ coincides with $-\varepsilon \Phi(c)$ we deduce that the
search of invariant curves is reduced to finding the zeros of
$\Phi$. 

\begin{rem}
Motivated by the quasiperiodically forced Arnold circle map~(\ref{Eq:qpfAmap}),
let us now consider the mapping
\begin{equation}
\label{Eq:FEOmega}
f_{\varepsilon,\omega_0,\omega_1}:\Cyl\To \Cyl\quad
\left\{\begin{array}{l}
r_1=r+\omega_0+\varepsilon\left(\omega_1+F(r,\theta;\varepsilon)\right)\\
\theta_1=\theta+2\pi\alpha\end{array}\right.
\end{equation}
where $F$ is as before and $\omega_0$, $\omega_1$ are real
parameters. For $\omega_0=0$ and $\varepsilon$ small but fixed, all
mappings $f_{\varepsilon,0,\omega_1}$ have the same family of
translated curves. Invariant curves will correspond to the roots of
$\Phi(c)+\omega_1=0$, with $\Phi$ given in
equation~(\ref{Eq:BifFunc}). If we now assume that $F$ is also
$2\pi$-periodic in $r$ (as in the quasiperiodically forced Arnold circle
map~(\ref{Eq:qpfAmap})), then $\Phi$ is $2\pi$-periodic in $c$. This implies that the
set of $\omega_1$'s for which there is at least one invariant curve,
that is, $\Phi^{-1} (-\omega_1)$ is non-empty, is a compact
interval. This is a fact that seems to be given for granted
in~\cite{GlendinningW99}. Moreover, if $[\omega_\ast,\omega^\ast]$
is such an interval, then there exist at least two invariant curves
if $\omega_1 \in (\omega_\ast,\omega^\ast)$. In principle, the
degenerate situation $\omega_\ast=\omega^\ast$ cannot be excluded.
In such a case, all translated curves would become invariant
simultaneously and the mapping $f_{\varepsilon,0,\omega^\ast}$
should be conjugate to $r_1=r$, $\theta_1=\theta+2\pi\alpha$.
\end{rem}

Next, we discuss the stability properties of invariant curves. Let
$r=\psi(\theta)$ be an invariant curve of the
mapping~(\ref{Eq:FE})
 \[
 f_{\varepsilon}:\Cyl\To \Cyl\quad
 \left\{\begin{array}{l}r_1=r+\varepsilon F(r,\theta;
 \varepsilon)\, ,\\ \theta_1=\theta+2\pi\alpha\, .\end{array}\right.
 \]
The change of variables $r=\xi+\psi(\theta)$ maps $r=\psi(\theta)$
onto $\xi=0$ and transforms the mapping~(\ref{Eq:FE}) to
 \[
 \begin{array}{l}
 \xi_1=\xi+\varepsilon\left(F(\xi+\psi(\theta),\theta;\varepsilon)
 -F(\theta,\psi(\theta);\varepsilon)\right)\, ,\\
 \theta_1=\theta+2\pi\alpha\, .\end{array}
 \]
The linearization around $\xi=0$ is
 \begin{equation}
 \label{Eq:VarEq}
 z_1=a(\theta;\varepsilon)z,\qquad
 \theta_1=\theta+2\pi\alpha ,
 \end{equation}
where
$a(\theta;\varepsilon)=1+\varepsilon F_r(\psi(\theta),\theta;\varepsilon)$.
The solution of the variational difference equation~(\ref{Eq:VarEq}) is
 \begin{equation}
 \label{Eq:SolVarEq}
 z_n=\left[\prod^{n-1}_{k=0}a(\theta+2\pi k\alpha;\varepsilon)\right]z_0\, .
 \end{equation}
In the terminology of the supplement of~\cite[page 660]{KatokH95},
this defines a cocycle over the
rotation $\theta_1=\theta+2\pi\alpha$. The Lyapunov exponent is defined as
 \[
 \LE(\theta)=\lim_{n\to\infty}\frac{1}{n}\sum^{n-1}_{k=0}
 \log a(\theta+2\pi k\alpha;\varepsilon)\, .
 \]
The Birkhoff Ergodic Theorem implies that this limit exists almost
everywhere and is independent of $\theta$; namely,
 \[
 \LE=\frac{1}{2\pi}\int^{2\pi}_0
 \log(1+\varepsilon F_r(\psi(\theta),\theta;\varepsilon)) \diff \theta\, .
 \]
Actually, it is possible to be more precise about the rate of
convergence, as shown by the following result.

\begin{lem}
 \label{Lem:Fourier} Let $f\in H^2(\Tset )$. Then
 \[
 \frac{1}{2\pi}\int^{2\pi}_0 f(\theta) \diff \theta
 =\frac{1}{n}\sum^{n-1}_{k=0}f(\theta+2\pi k\alpha)
 + \Or{\frac{1}{n}}
 \]
as $n\to \infty$, uniformly in $\theta$.
\end{lem}
\begin{rem}
When the order of convergence does not need to be specified, this is a very classic lemma that appears, for example, in \cite[p.~369]{NemytskiiS60}, and just Riemann integrable $f$ and irrational $\alpha$ are required.
\end{rem}
    
\proof
Since $\alpha$ is of constant type~(\ref{Eq:ConsType}), one can obtain an
estimate of the type
 \[
 \abs{\e ^{2\pi \ic m \alpha}-1}\geq \frac{1}{\mu\abs{m}}\, ,
 \quad m\neq 0,
 \]
where $\mu >0$ only depends on $\alpha$. Thus, for $m\neq 0$,
trigonometric sums can be estimated as follows,
 \[
 \abs{\sum_{k=0}^{n-1} \e ^{2\pi \ic k m \alpha}}
 \leq \frac{2}{\abs{\e ^{2\pi \ic  m \alpha}-1}}
 \leq 2\mu \abs{m} .
 \]
Let $\sum \widehat f _m \e^{ \ic m\theta}$ be the Fourier expansion
of $f(\theta )$. To prove this lemma, one only has to check that
$\sum^{n-1}_{k=0}\left(f(\theta+2\pi k\alpha)-\widehat f_0\right)
<\infty$. In view of the previous estimates it is enough to check
that $\sum \abs{\widehat f _m}\abs{m}<\infty$. This series is
convergent because $f\in H^2 (\Tset )$ and
 \[
 \sum_{m\neq 0}\abs{\widehat f _m} \abs{m}
 \leq\left(\sum \abs{m}^{-2}\right)^{1/2}
 \left(\sum \abs{\widehat f _m}^2 \abs{m}^4 \right)^{1/2}
 <\infty .\] 
 Note that $$\sum^{n-1}_{k=0}\left(f(\theta+2\pi k\alpha)-\widehat f_0\right)=\sum_{m\neq 0} \widehat f_m e^{\ic m\theta} \sum_{k=0}^{n-1} e^{2\pi \ic km\alpha}.\mqed$$
Our invariant curve $r=\psi(\theta)$ is in $H^2(\Tset )$ and so is
$a(\theta;\varepsilon)$ in equation~(\ref{Eq:VarEq}). Thus
lemma~\ref{Lem:Fourier} applies to $\log a$ and one obtains the following
estimate
 \[
 \abs{z_n}\leq Ke^{n\LE} \abs{z_0}\, ,\quad n\geq 1\, ,
 \]
for the solution~(\ref{Eq:SolVarEq}) of the variational difference
equation~(\ref{Eq:VarEq}).

It is now easy to prove that if $\LE<0$ then $\xi=0$ is a local
attractor for the nonlinear map. Actually, one uses standard
arguments in linearization. Analogously, $\xi=0$ is a repeller if
$\LE> 0$. The case $\LE=0$ is not decided by the first approximation.

\begin{figure}[ht]
 \includegraphics[width=0.49\textwidth]{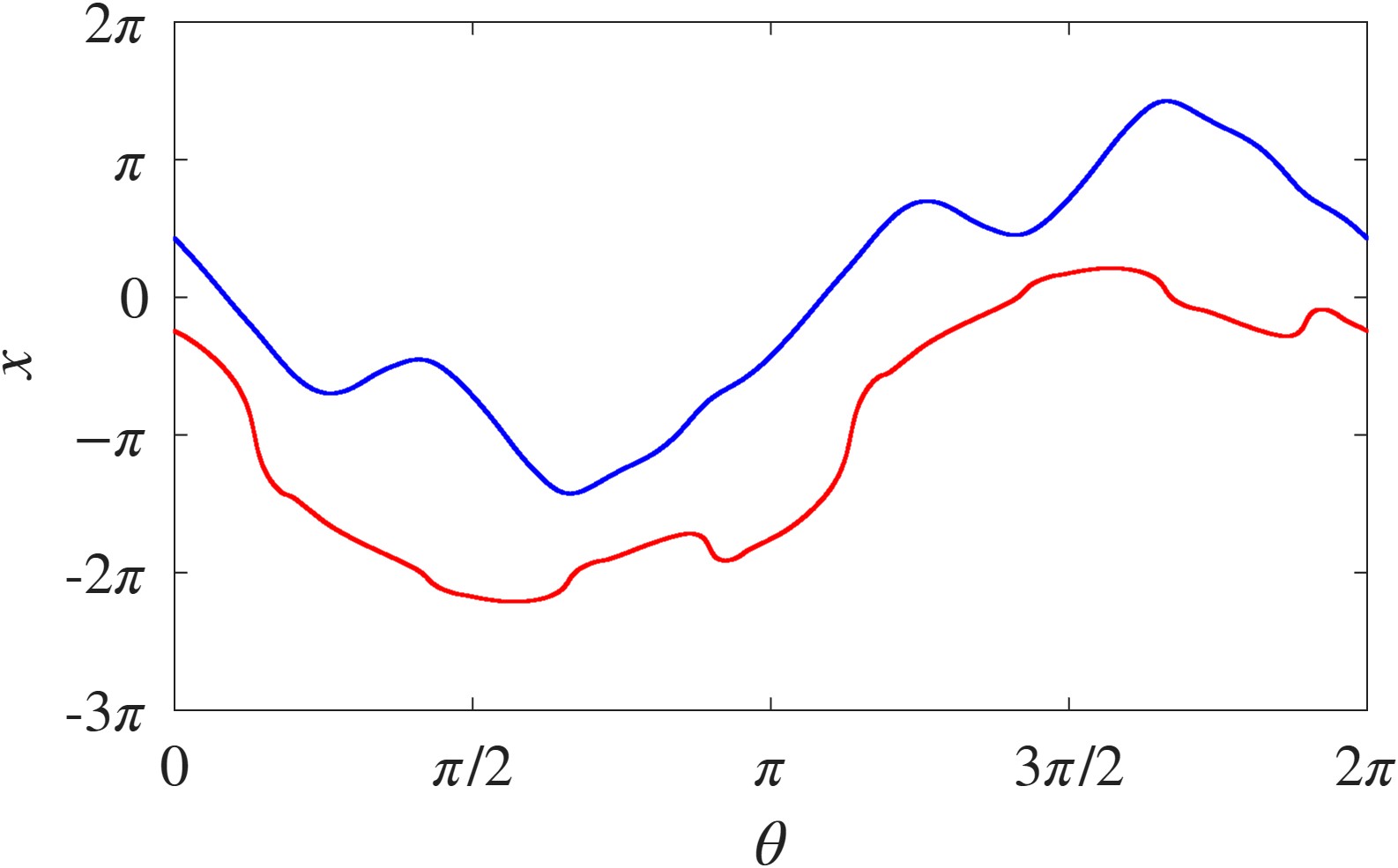}
  \includegraphics[width=0.49\textwidth]{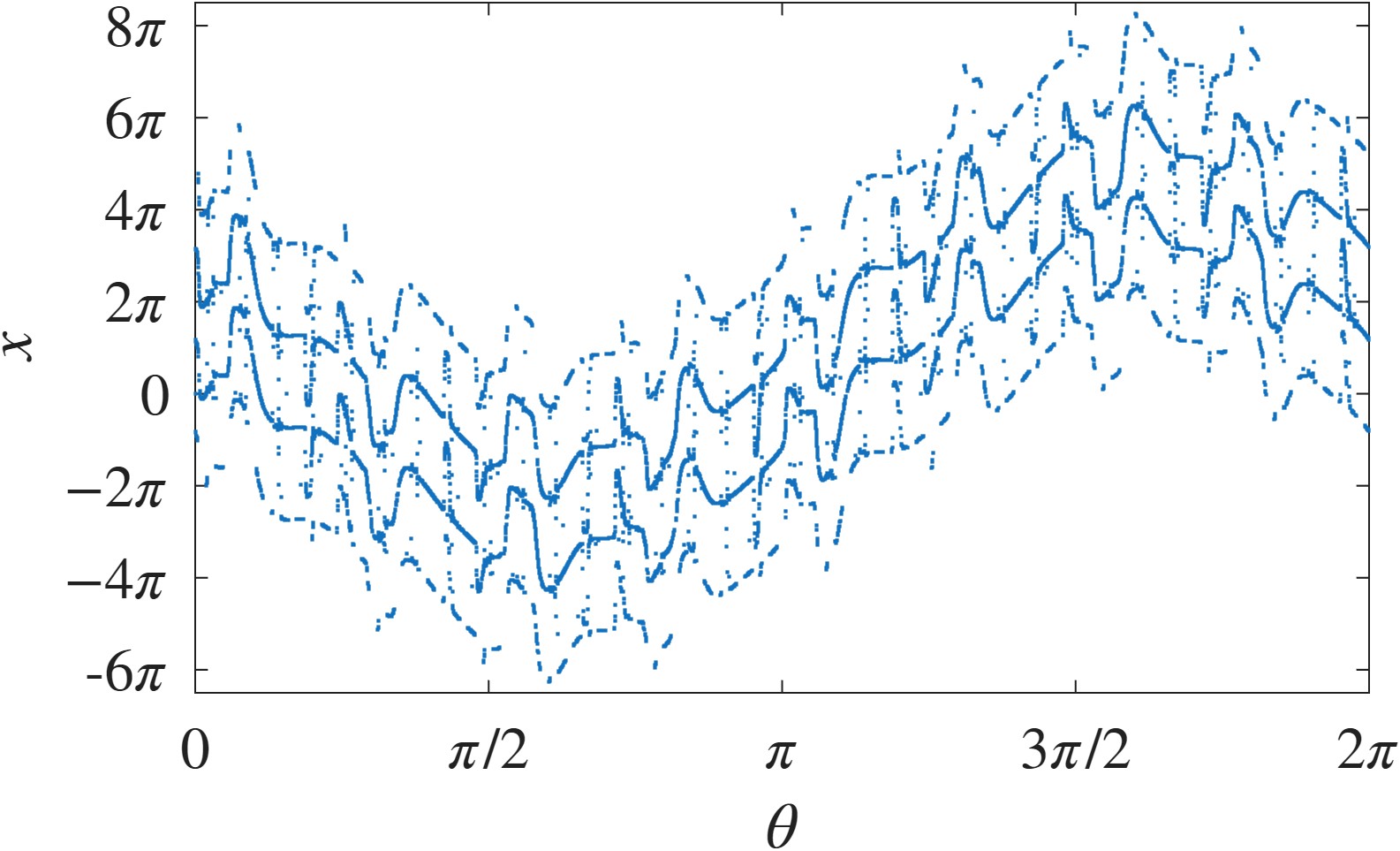}
 \caption{For the quasiperiodically forced Arnold circle map~(\ref{Eq:qpfAmap}) with  $k=0.8$, $\omega=0$, $\alpha=\frac{\sqrt{5}-1}{2}$, there appear a blue attractor invariant curve and a red repeller one for $\displaystyle\frac{b}{2\pi}=1.1$, and a possible SNA for $\displaystyle\frac{b}{2\pi}=3.1$.}
 \label{Fig:InvCurveAndSNA}
 \end{figure}
 
When $\varepsilon$ is small, it is possible to connect the sign of
$\LE$ with the sign of the derivative of $\Phi$ at a zero
$\displaystyle c^\ast=1/(2\pi)\int^{2\pi}_0\psi$ of $\Phi$. In fact,
 \[
 \LE=\frac\varepsilon{2\pi}\int^{2\pi}_0
 F_r(\psi_{c^\ast}(\theta),\theta;\varepsilon) \diff \theta+
 \Or{\varepsilon^2}
 \]
while
 \begin{eqnarray*}
 \Phi'(c^\ast)&=&\frac{1}{2\pi}\int^{2\pi}_0
 F_r(\psi_{c^\ast}(\theta),\theta;\varepsilon)
 \left.\Dc \psi\right|_{c=c^\ast}(\theta)\, \diff \theta\\
 &=&\frac{1}{2\pi}\int^{2\pi}_0
 F_r(\psi_{c^\ast}(\theta),\theta;\varepsilon) \diff \theta
 +\Or{\varepsilon}\, .
\end{eqnarray*}
In the last identity we used the estimate $\dc \psi (\theta
)=1+\Or{\varepsilon}$ (see~(\ref{otra})). We can continue a step
further and, assuming that $F$ is smooth with respect to
$\varepsilon$, prove that $\psi$ is smooth in $\varepsilon$. The
Implicit Function Theorem allows us to obtain the following
existence result for the case of hyperbolic invariant curves. (More
on this can be found in section~\ref{SubsecAveraging}.)

\begin{thm}
Given the mapping~\textup{(\ref{Eq:FE})}, consider the function
 \[
 \Phi_0(c):=\frac1{2\pi}\int^{2\pi}_0 F(c,\theta;0) \diff \theta\, .
 \]
Then each simple zero $c^\ast$ of $\Phi_0$ gives rise, for
$\varepsilon $ small enough, to a unique hyperbolic invariant curve
of the mapping~\textup{(\ref{Eq:FE})}, whose Lyapunov exponent is
given by $\LE=\varepsilon\Phi'_0 (c^{\ast} ) + \Or{\varepsilon^2}$.
\end{thm}

\begin{rem}
Of course, multiple zeros of $\Phi_0$ of odd order give
also rise, for $\varepsilon $ small enough, to invariant curves of
the mapping~(\ref{Eq:FE}), but the uniqueness and the hyperbolicity
character may not be preserved. In the case of a function $F$ depending
$2\pi$-periodically on the variable $x$ and with zero mean value
(like the quasiperiodically forced Arnold circle map~(\ref{Eq:qpfAmap})),
$\Phi_0$ will typically have at least two simple zeros,
giving rise at least to an attractor and a repeller for
the mapping~(\ref{Eq:FE}).
\end{rem}
\section{Final remarks}
\label{Sec:FinalRem}

\subsection{Case of a rational $\alpha$}
\begin{figure}[ht]
\centering
\includegraphics[width=0.6\textwidth]{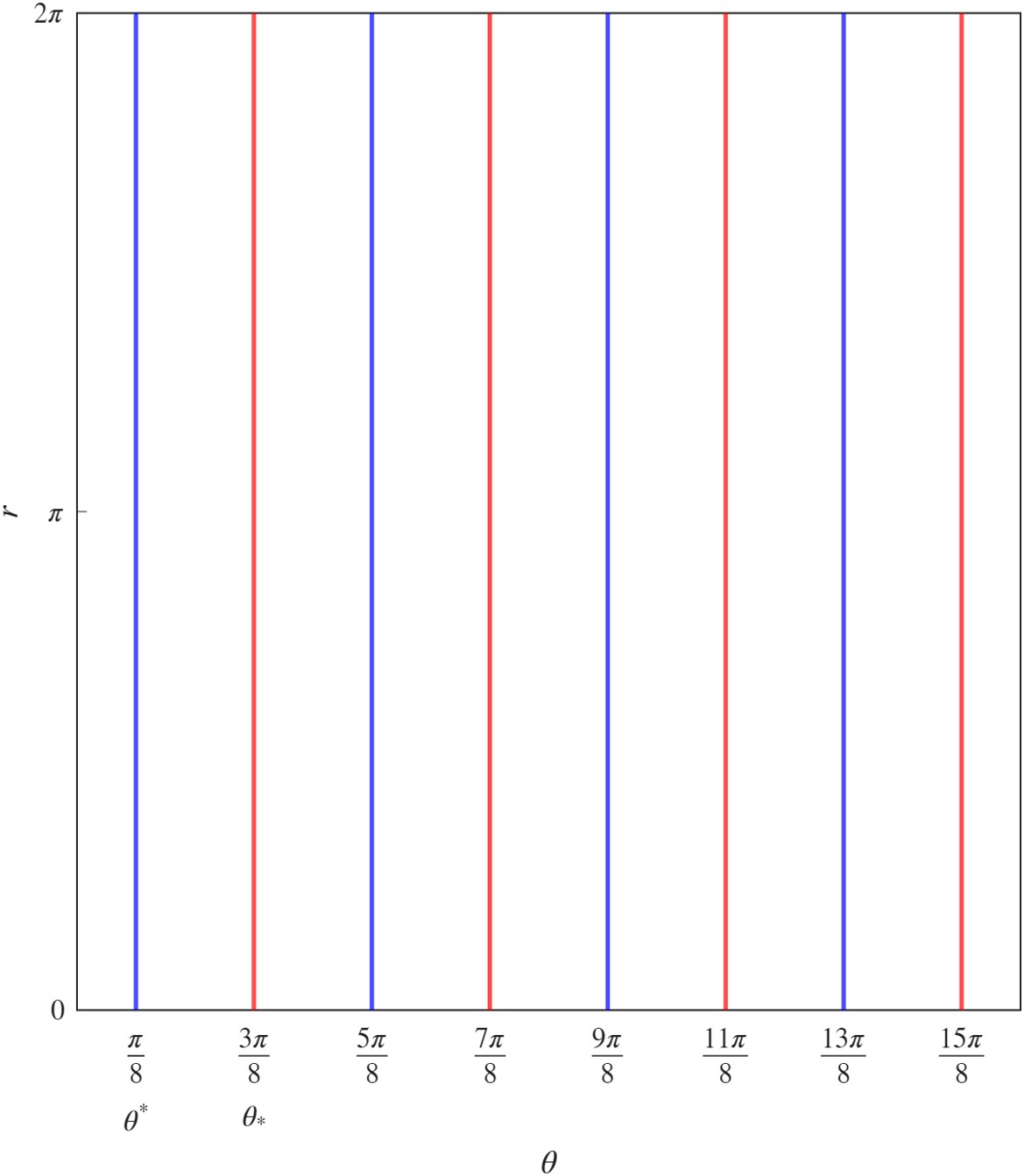}
\caption{Radial action for $F(r,\theta)=(1+\sin^2r)\sin 4\theta$
on $\theta=\frac{\pi}{8},\dots, \frac{15\pi}{8}$. The dynamics is upwards on the set of blue lines, and downwards on the red ones.}
\label{Fig:RatAlpha}
\end{figure}
We give an example showing that the existence of translated curves
can fail when $\alpha\in\Qset $. When $\alpha$ is of Liouville type
there are examples in~\cite{KatokH95} for $F=F(\theta)$, as was
already mentioned in section~\ref{Sec:translated}.

Assume $\alpha=p/q$ in lowest terms and that there exist two angles
$\theta^\ast$, $\theta_\ast$ such that
\[
F(r,\theta_\ast+2\pi k\alpha)<0\, ,
\quad F(r,\theta^\ast+2k \pi\alpha)>0\, ,\quad\forall\, r\in\Rset \, ,
k=0,1,\dots,q-1.
\]
For example, consider $F(r,\theta)=(1+\sin^2r)\sin q \theta$
(see figure~\ref{Fig:RatAlpha}).

Adding from $k=0$ to $q-1$ in the functional
equation~(\ref{Eq:FuncEq}) we are led to
 \begin{eqnarray*}
 0&=&\sum^{q-1}_{k=0}[\psi(\theta+2\pi(k+1)\alpha)
 -\psi(\theta+2\pi k\alpha)]\\ &=&\varepsilon\sum^{q-1}_{k=0}
 F(\psi(\theta+2\pi k\alpha),\theta+2\pi k\alpha;\varepsilon)+q
 \lambda\, .
 \end{eqnarray*}
Letting $\theta=\theta_\ast$ and $\theta=\theta^\ast$
we find a contradiction on the sign of $\lambda$.

\subsection{On the numbers $\nu_n$}
\label{SubsecNun} The value $\varepsilon^\ast$ in
Theorem~\ref{Thm:Main} about translated curves can be determined
(although not optimally). For this, it would be enough to use a
quantitative version of lemma~\ref{Lem:Herman} as in~\cite{Herman86}.

To justify the need of the numbers $\nu_n$ in the iteration
scheme~(\ref{Eq:DefPsin}) assume that, at some step $n$, we can find
a positive function $b(\theta)$ such that
\begin{eqnarray*}
1+\varepsilon F_r(\psi_n(\theta),\theta;\varepsilon)
=b(\theta)/b(\theta+2\pi\alpha),\\
\int^{2\pi}_0 F_\theta(\psi_n(\theta),\theta;\varepsilon)\,
b(\theta+2\pi\alpha)\diff \theta\not=0\, .
\end{eqnarray*}
Then we cannot solve
\begin{eqnarray*}
\D \psi _{n+1}(\theta+2\pi\alpha)
&=&\Bigl(1+\varepsilon F_r(\psi_n(\theta),\theta;\varepsilon)\Bigr)
\D \psi _{n+1}(\theta)\\
&&\mbox{}+\varepsilon F_\theta(\psi_n(\theta),\theta;\varepsilon),
\end{eqnarray*}
because if we multiply  by $b(\theta+2\pi\alpha)$ and integrate over
a period we arrive at a contradiction:
\begin{eqnarray*}
0&=&\int^{2\pi}_0 b(\theta+2\pi\alpha)
\D \psi_{n+1}(\theta+2\pi\alpha)\diff \theta
-\int^{2\pi}_0 b(\theta)\D \psi_{n+1}(\theta) \diff \theta\\
&=&\varepsilon\int^{2\pi}_0
b(\theta+2\pi\alpha)\, F_\theta(\psi_n(\theta),\theta;\varepsilon)\diff \theta\neq 0
\, .
\end{eqnarray*}

\subsection{An alternative approach to the existence of invariant curves}
\label{SubsecAveraging} Assume that $F$ is smooth in $\varepsilon$
and define the average of $F$ as
$\overline{F}(r;\varepsilon)=1/(2\pi)\int^{2\pi}_0
F(r,\theta;\varepsilon)\diff \theta$ and $\widetilde F=
F-\overline{F}$. Let $H=H(r,\theta;\varepsilon)$ be a solution of
the small divisors equation
\begin{equation}
H(r,\theta+2\pi\alpha;\varepsilon)-H(r,\theta;\varepsilon)
+\widetilde F(r,\theta;\varepsilon)=0\, .
\label{Eq:SmallDiv}
\end{equation}
The change of variables $\rho=r+\varepsilon H(r,\theta;\varepsilon)$
transforms the mapping $f_\varepsilon$ given in~(\ref{Eq:FE}) into
\[
\begin{array}{l}
\rho_1=\rho+\varepsilon\bar
F(\rho,\varepsilon)+\Or{\varepsilon^2}\\
\theta_1=\theta+2\pi\alpha\, .\end{array}
\]
Let $\rho^\ast\in\Rset $ be such that $\overline{F}(\rho^\ast)=0$
and $\overline{F}'(\rho^\ast)\not=0$.
Then one can apply the theory of normally hyperbolic
invariant manifolds, with $\Lambda:\rho=\rho^\ast$, and obtain an
invariant curve $\Lambda_\varepsilon$ for small $\varepsilon$. This
approach requires hyperbolicity while ours does not. For example,
in the map $f_{\varepsilon,0,\omega_1}$ given in~(\ref{Eq:FEOmega})
there is no hyperbolicity if $\omega_1=\omega_\ast$ or $\omega^\ast$.
We notice  that in this approach one also needs arithmetic conditions
on $\alpha$, they appear when one tries to solve
equation~(\ref{Eq:SmallDiv}) equation for $H$. Typically,
the function $H$ will be less smooth than $F$.

\subsection{The mapping on the torus}

Assume that, as in the quasiperiodically forced Arnold circle
map~(\ref{Eq:qpfAmap}), $F$ is $2\pi$-periodic also in $r$.
Then we can think of $f_\varepsilon$ as a diffeomorphism of
the torus $\Tset ^2=\Tset \times\Tset $. Now the
translated curves will satisfy $\psi_{c+2\pi}=\psi_c+2\pi$ but still
they are a foliation of $\Tset ^2$. Invariant curves appear when the
translation number $\lambda$ is in $2\pi\Zset $. Following section 5 of
\cite{Herman83} we can define the fibred rotation number of
$f_\varepsilon$ by
\[
\rho=\lim_{n\to\infty}\frac1n(r_n-r_0)\quad\mbox{where}\quad
(r_n,\theta_n)=f^n_\varepsilon (r_0,\theta_0)\, .
\]
This number is independent of $(r_0,\theta_0)$ and it only depends
on the lift of $f_\varepsilon$. Let $r=\psi(\theta)$ be an
invariant curve with
\[
\psi(\theta+2\pi\alpha)=\psi(\theta)
+\varepsilon F(\psi(\theta),\theta;\varepsilon)+2\pi N\, .
\]
Then $\rho=-2\pi N$ and the dynamics on these curves is called
\emph{mode-locked}. When $\varepsilon$ is small, the only
possibility is $N=0$.


Let us now consider $f_{\varepsilon,\omega_0,\omega_1}$ as given
in~(\ref{Eq:FEOmega}). For each $N\in\Zset $, imposing $\omega_0=N$,
we find a compact interval $I_N=[\omega_\ast(N),\omega^\ast(N)]$
such that $f_{\varepsilon,\omega_1}$ has an invariant curve with
$\lambda=2\pi N$ if and only if
$\omega_1\in[\omega_\ast(N),\omega^\ast(N)]$. These intervals must
be disjoint and it seems that on the torus the invariant curves
could present a phenomenon of intermittency with respect to the
parameter $\omega_1 \in \Tset$.


\subsection{Basin of attraction}
Let $r=\psi(\theta)$ be an invariant curve with $\LE<0$. We know
from section \ref{Sec:ExistStab} that this curve is a local attractor.
Can we say that the boundary of the region of attraction is made
of invariant curves? This is unclear. In principle we cannot exclude
the existence of invariant sets which are not invariant curves.


This will probably happen when $\alpha$ is a Liouville number and
there are no invariant curves, but it might also happen when
$\alpha$ is of constant type. These invariant sets would probably
be the closure of graphs of $r=\psi(\theta)$ with $\psi$ a
discontinuous solution of the functional
equation~(\ref{Eq:FuncEq}). This is reminiscent of Aubry-Mather
sets and also of some of the phenomena described in section 4 of
\cite{Herman83}.

\section*{Acknowledgments}
We are grateful to R.~de la Llave for several comments and suggestions.
A.D. is partially supported by the grant PID-2024-158570NB-100 funded by MCIN/\allowbreak AEI/\allowbreak 10.13039/\allowbreak 501100011033 and ``ERDF A
way of making Europe''. R.O. is partially supported by the grant  PID2021-128418NA-I00. The first version was prepared while both
authors were visitors to the \emph{Centre de Recerca Matem\`{a}tica}, for
whose hospitality they are very grateful.
\bibliographystyle{alpha}
\bibliography{qpfcm}
\end{document}